\documentclass[12pt,leqno]{article}
\usepackage{amssymb,amsmath,amsthm}
\usepackage[all]{xy}
\usepackage[alphabetic,lite]{amsrefs}
\numberwithin{equation}{section}

\usepackage[top=2cm,bottom=2cm,left=2cm,right=4cm,marginparsep=0.3cm,marginparwidth=3cm,includefoot]{geometry}
\usepackage{mathtools}
\usepackage{enumerate}
\usepackage{mathrsfs}
\usepackage{dist}

\usepackage[colorlinks=true, pdfstartview=FitV, linkcolor=blue, citecolor=blue, urlcolor=blue]{hyperref}

\begin{document}

\title
{A property of the interleaving distance for sheaves}

\author{Fran{\c c}ois Petit, Pierre Schapira and Lukas Waas}
\maketitle
  
\begin{abstract}
Let  $X$ be a real analytic manifold endowed with a  distance satisfying suitable properties and let $\cor$ be a field. In~\cite{PS20}, the authors construct a pseudo-distance  on the derived category of sheaves of $\cor$-modules on $X$, generalizing a previous construction of~\cite{KS18}. 
We prove here that if the distance between two constructible sheaves with compact support (or more generally, constructible sheaves up to infinity) on $X$ is zero, then these two sheaves are isomorphic.
This answers in particular a question of~\cite{KS18}.
\end{abstract}

{\renewcommand{\thefootnote}{\mbox{}}
	\footnote{The research of F.P.  was supported by the IdEx Université de Paris, ANR-18-IDEX-0001}
	\footnote{The research of L.W. was supported by the Landesgraduiertenförderung Baden-Württemberg.}
	\addtocounter{footnote}{-4}
}

\section{Introduction}

The interleaving distance was introduced in \cite{FCMGO09} and provides a pseudo-metric on the category of persistent modules. It was generalized to multi-persistence modules by M. Lesnick in \cites{Les12, Les15}. In his thesis \cite{Cur14}, J. Curry  showed how to interpret the notion of persistent modules in the classical language of sheaves. This allows one to interpret the interleaving pseudo-distance as a pseudo-metric on the category of $\gamma$-sheaves on a finite-dimensional real vector space where $\gamma$ is a convex proper cone. In \cite{KS18}, M. Kashiwara and P. Schapira systematically treat persistent homology in the framework of \textit{derived} sheaf theory. In particular, they introduce the so-called convolution pseudo-distance for derived sheaves on real normed vector spaces. There, they asked if this pseudo-distance is a distance \cite{KS18}*{Rem.~2.3}, that is if two sheaves with convolution distance zero are necessarily isomorphic. A similar question had been studied by Lesnick in \cite{Les15} where he proved that the interleaving pseudo-distance restricted to finitely presented persistent modules is a distance. This result is related to the aforementioned question as it follows from \cite{BP21} that the restriction of the convolution pseudo-distance to the subcategory of $\gamma$-sheaves is equal to the interleaving pseudo-distance. In \cite{BG18}, N. Berkouk and G. Ginot proved that the convolution pseudo-distance is in general not a distance and established, by constructing an appropriate matching distance, that it is a distance on constructible sheaves on $\R$.
Here, following \cite{PS20}, we consider a generalization of the convolution pseudo-distance for sheaves on ``good'' metric spaces (as for instance complete Riemannian manifolds with strictly positive convexity radius). We prove that on a real analytic manifold endowed with a ``good'' distance, the pseudo-distance on sheaves is a distance when restricted to constructible sheaves with compact support (or more generally, constructible sheaves up to infinity).

\section{Review}
Throughout this paper, $\cor$ is a field. We shall mainly follow the notations of~\cite{KS90} for sheaf theory.
For a topological space $X$ we denote by $\Delta$ the diagonal of $X\times X$  and by $a_X$ the map $X\to\rmpt$.
 We denote by $\Derb(\cor_X)$ the bounded derived category of sheaves of $\cor$-modules. If $X$ is a real analytic manifold, we denote by  $\Derb_\Rc(\cor_X)$ the full triangulated subcategory of  $\Derb(\cor_X)$ consisting of $\R$-constructible sheaves. We denote by $\omega_X$ the dualizing complex.

\subsection{Kernels}
Recall that a topological  space $X$ is good if it is Hausdorff, locally compact, countable at infinity and of finite flabby dimension.

Given topological spaces $X_i$ ($i=1,2,3$) we set
$X_{ij}=X_i\times X_j$, $X_{123}=X_1\times X_2\times X_3$.    We denote by $q_i\cl X_{ij}\to X_i$
and $q_{ij}\cl X_{123}\to X_{ij}$ the projections.

For $A_{ij}\subset X_{ij}$ with $i=1,2$, $j=i+1$, one defines $A_{12}\conv A_{23}\subset X_{13}$ as
\eq\label{eq:convker0}
&&A_{12}\cconv[2] A_{23}\eqdot {q_{13}}(\opb{q_{12}}A_{12}\cap\opb{q_{23}}A_{23}).
\eneq
For good topological spaces $X_i$'s, one often calls an object $K_{ij}\in\Derb(\cor_{X_{ij}})$ {\em a kernel}.
One defines  as usual the  composition of kernels by
\eq\label{eq:convker}
&&K_{12}\cconv[2] K_{23}\eqdot \reim{q_{13}}(\opb{q_{12}}K_{12}\ltens\opb{q_{23}}K_{23}).
\eneq
If there is no risk of confusion, we write $\conv$ instead of $\cconv[2]$. 

\subsection{Distances}

For a metric space $(X,d)$, $x_0\in X$ and $a\in\R_{\geq0}$, we set
\begin{equation*}
	\begin{array}{ll}
		B_a(x_0)=\{x\in X;d(x_0,x)\leq a\}, & B_a^\circ(x_0)=\{x\in X;d(x_0,x) < a\}\\
		\Delta_a=\{(x,y)\in X\times X; d(x,y)\leq a\}, & \Delta^\circ_a=\{(x,y)\in X\times X; d(x,y)< a\},\\
		\Delta^+=\{(x,y,t)\in  X\times X\times\R;d(x,y)\leq t\}, & \Delta^{+, \circ}=\{(x,y,t)\in  X\times X\times\R;d(x,y) < t \}.
	\end{array}
\end{equation*}

Following~\cite{PS20}, we say that a metric space $(X,d)$ is good, or simply that the distance is good,  if the underlying topological space is good and moreover 
\eq\label{hyp:dist1}
&&\left\{\parbox{70ex}{
 there exists some  $\alpha_X>0$
such that for all $0\leq a,b$ with $a+b\leq\alpha_X$:\\
(i) for any $x_1,x_2\in X$, $B_a(x_1)\cap B_b(x_2)$ is contractible or empty (in particular, for any $x\in X$, $B_a(x)$ is contractible),\\
(ii) the two projections $q_1$ and $q_2$ are proper on $\Delta_a$,\\
(iii) $\Delta_a\conv\Delta_b=\Delta_{a+b}$.
}\right.
\eneq
In order to construct a distance on sheaves in this situation, the main idea of~\cite{PS20} is to use the kernel $\cor_{\Delta_a}$ when
$0\leq a<\alpha_X$ and to replace it by a composition of kernels $\cor_{\Delta_b}$ with  $0\leq b<\alpha_X$ otherwise.

We shall also consider the hypotheses~\eqref{hyp:dist2} below which insure that  the kernels $\cor_{\Delta_a}$ are invertible (see~\cite{PS20}*{Prop.~2.2.3}). 

Let $U$ be an open  subset of a real $C^0$-manifold $M$. 
We  say that  $U$  is locally topologically convex  (l.t.c. for short)  in $M$ if each $x\in M$ admits an open neighborhood $W$ such that there exists a topological isomorphism $\phi\cl W\isoto V$, with $V$ open in a real vector space, such that $\phi(W\cap U)$ 
is convex.
Clearly, if $U$ is  l.t.c. then it is l.c.t.

\eq\label{hyp:dist2}
&&\left\{\parbox{65ex}{
	The good metric  space  $X$ is a $C^0$-manifold and 
	\banum
	\item
	for   $0<a\leq\alpha_X$, the set $\Delta_a^\circ$ is l.t.c. in $X\times X$,
	\item
	the set $\Delta^{+,\circ}$ is l.t.c.  in $X\times X\times]-\infty,\alpha_X[$.
	\item
         For  $x,y\in X$, setting $Z_{a}(x,y)=B_a(x)\cap B^\circ_a(y)$, one has $\rsect(X;\cor_{Z_{a}(x,y)})\simeq0$ for $x\neq y$ and  $0<a\leq\alpha_X$. 
        \eanum
}\right.
\eneq

%
 
In loc.\ cit.\ it is shown that complete Riemannian manifolds with strictly positive convexity radius as well as normed vector spaces satisfy conditions \eqref{hyp:dist1} and  \eqref{hyp:dist2}. Moreover, given a good metric space $(X,d)$, one can naturally associate a pseudo-distance on the objects of the derived category $\Derb(\cor_X)$, generalizing the convolution distance first introduced in~\cite{KS18}.

\subsection{Constructible sheaves}
Here,  we work in the framework of real
analytic manifolds and use the notions of being subanalytic or constructible ``up to infinity'', following~\cite{Sc20}.

 Recall that  a b-analytic manifold $\fX=(X,\bX)$ is the data of  real analytic manifold $\bX$ and a subanalytic open relatively compact subset $X\subset \bX$. One denotes by $j_X\cl X\into \bX$ the embedding. A morphism of  b-analytic manifolds is a morphism of real analytic  manifolds $f\cl X\to Y$ whose graph is subanalytic in $\bX\times\bY$. 

One defines naturally  the full triangulated subcategory $\Derb_\Rc(\cor_{\fX})$ of $\Derb_\Rc(\cor_X)$  consisting of sheaves constructible up to infinity. These are the objects $F$ of $\Derb_\Rc(\cor_{X})$ such that $\eim{j_X}F\in\Derb_\Rc(\cor_{\bX})$. 
The advantage of the notion of being constructible up to infinity is that it is stable under the six operations, in particular, by direct images
of morphisms of b-analytic manifolds \cite{Sc20}. 

In the sequel, instead of writing ``subanalytic up to infinity '' or ``constructible  up to infinity'', we shall  write 
``b-subanalytic'' or ``b-constructible''.

Let $F\in\Derb_\Rc(\cor_X)$. 
Recall that  $F$ is b-constructible  if $\eim{j_X}F\in\Derb_\Rc(\cor_{\bX})$. One denotes by 
$\Derb_\Rc(\cor_{\fX})$ the full triangulated subcategory of $\Derb_\Rc(\cor_X)$  consisting of sheaves b-constructible.

\section{Main theorem}
In this section, $\fX$ is a b-analytic manifold endowed with a good distance.

Recall the definition of being $a$-isomorphic and the associated pseudo-distance $\dist$, following~\cite{KS18} generalized in~\cite{PS20}.

\begin{definition}\label{def:1}
Let $F,G\in\Derb(\cor_X)$ and let $a>0$. One says that $F$ and $G$ are $a$-isomorphic if there exist 
morphisms $u_a\cl F\conv  \cor_{\Delta_a}\to G$, $v_a\cl G\conv  \cor_{\Delta_a}\to F$ such that 
$F\conv  \cor_{\Delta_{2a}}\to G\conv  \cor_{\Delta_a}\to F$ and $G\conv  \cor_{\Delta_{2a}}\to F\conv  \cor_{\Delta_a}\to G$ are 
the natural morphisms  associated with $ \cor_{\Delta_{2a}}\to\cor_\Delta$.

One sets 
\eqn
&&\dist(F,G)=\inf\{a\in[0,+\infty];\text{$F$ and $G$ are $a$-isomorphic}\}.
\eneqn
\end{definition}
In the sequel, when considering thickenings of the diagonal $\Delta_a$, we shall assume $0\leq a<\alpha_X$, where $\alpha_X$ is given in~\eqref{hyp:dist1}. For $0\leq a\leq b<\alpha_X$, the morphism
\eqn\label
&&\cor_{\Delta_b}\to\cor_{\Delta_a}
\eneqn
define the morphisms
\eq
\RHom(F\conv \cor_{\Delta_a},G)&\to&\RHom(F\conv \cor_{\Delta_b},G),  \label{eq:mor1} \\
\RHom(F,G\conv \cor_{\Delta_b})&\to&\RHom(F,G\conv \cor_{\Delta_a}) \label{eq:mor2}.
\eneq
One of the central step of the proof is to establish that the morphisms \eqref{eq:mor1} and \eqref{eq:mor2} are isomorphisms. For that purpose, we will prove that these morphisms spaces can be written as the cohomology of a constructible sheaf on $\R$ supported by balls of radius $a$ and $b$. The constructibility will imply that, for $0 \leq a \leq b$ sufficiently small, these cohomology groups are isomorphic which implies the result.

Denote by $q_1\cl  X\times\R\to X$ and  $q_2\cl  X\times\R\to \R$  the projections.
Also denote by $I_a$ the closed interval $[-a,a]$ of $\R$ and by $I_a^\circ$ the open interval.
\begin{lemma}\label{le:2}
Let $F,G\in\Derb_\Rc(\cor_{\fX})$.  
\bnum
\item For $a\geq0$ one has, 
\eqn
\rsect_{I_a}(\R;\roim{q_2}\rhom(F\conv \cor_{\Delta^+},\epb{q_1}G))&\simeq &\RHom(F\conv \cor_{\Delta_a},G).
\eneqn

\item For $a>0$ one has,
\eqn
\rsect(I^\circ_a;\roim{q_2}\rhom(F\conv \cor_{\Delta^{+,\circ}},\epb{q_1}G))&\simeq &\RHom(F \conv \cor_{\Delta^\circ_a},G)[1].
\eneqn
\enum

\end{lemma}
\begin{proof}
(i) Set $K\eqdot \rhom(F\conv \cor_{\Delta^+},\epb{q_1}G)$. Then
\eqn
\rsect_{I_a}\roim{q_2}K&\simeq&\roim{q_2}\rsect_{\opb{q_2}I_a}K\\
&\simeq&\roim{q_2}\rhom((F\conv \cor_{\Delta^+})\tens\cor_{\opb{q_2}I_a},\epb{q_1}G)).
\eneqn
Therefore,
\eqn
\rsect_{I_a}(\R;\roim{q_2}K)&\simeq &\roim{a_X}\roim{q_1}\rhom((F\conv \cor_{\Delta^+})\tens\cor_{\opb{q_2}I_a},\epb{q_1}G))\\
&\simeq &\roim{a_X}\rhom(\reim{q_1}\bl (F\conv \cor_{\Delta^+})\tens\cor_{\opb{q_2}I_a}\br,G).
\eneqn
To conclude, let us check the isomorphism
\eq\label{eq:iso1}
&&\reim{q_1}\bl (F\conv \cor_{\Delta^+})\tens\cor_{\opb{q_2}I_a}\br\simeq F\conv \cor_{\Delta_a}.
\eneq
Consider the diagram
\eqn
&&\xymatrix{
&X\times X\times\R\ar[ldd]_-{p_1}\ar[d]|-{p_{12}}\ar[drr]^-{p_{23}}&&&\\
&X\times X\ar[ld]^-{r_1}\ar[rd]_-{r_2}&&X\times\R\ar[ld]^-{q_1}\ar[rd]_-{q_2}&\\
X&&X&&\R
}\eneqn
One has
\eqn
\reim{q_1}\bl (F\conv \cor_{\Delta^+})\tens\cor_{\opb{q_2}I_a}\br&\simeq &
\reim{q_1}\reim{p_{23}}\bl \opb{p_1}F\tens \cor_{\Delta^+}\tens\opb{p_{23}}\cor_{\opb{q_2}I_a}\br\\
&\simeq&\reim{r_2}\reim{p_{12}}\bl\opb{p_{12}}\opb{r_1}F\tens \cor_{\Delta^+\cap\opb{p_{23}}\opb{q_2}I_a})\br\\
&\simeq&\reim{r_2}\bl \opb{r_1}F\tens \reim{p_{12}} \cor_{\Delta^+\cap\opb{p_{23}}\opb{q_2}I_a}\br.
\eneqn
We now remark that $\Delta^+\cap\opb{p_{23}}\opb{q_2}I_a \subset \opb{p_{12}}(\Delta_a)$ and $p_{12}$ restricted to $\Delta^+\cap\opb{p_{23}}\opb{q_2}I_a$ is proper. Hence, there  are natural morphisms (the first morphism is obtained by adjunction)
\eq\label{eq:iso2}
\cor_{\Delta_a} \to \roim{p_{12}} \opb{p_{12}}\cor_{\Delta_a} \to \roim{p_{12}}\cor_{\Delta^+\cap\opb{p_{23}}\opb{q_2}I_a} \isofrom \reim{p_{12}}\cor_{\Delta^+\cap\opb{p_{23}}\opb{q_2}I_a}
\eneq
which define
\eq\label{eq:iso2b}
&&\cor_{\Delta_a} \to \reim{p_{12}}\cor_{\Delta^+\cap\opb{p_{23}}\opb{q_2}I_a}.
\eneq
It remains to prove that this last morphism is an isomorphism. One has $p_{12}(\Delta^+\cap\opb{p_{23}}\opb{q_2}I_a)=\Delta_a$ and the fibers of $p_{12}$ above $(x,y)\in X\times X$ is the interval 
$[d(x,y),a]$  which is contractible or empty. This, together with \cite{KS90}*{Prop. 2.5.2}, proves that~\eqref{eq:iso2b} is an isomorphism, hence proves~\eqref{eq:iso1}.

\spa
(ii) Replacing $I_a$ with $I_a^\circ$ and $\Delta^+$ with $\Delta^{+,\circ}$ in the proof of (i) it remains to show
\eq\label{eq:iso1open}
&&\reim{q_1}\bl (F\conv \cor_{\Delta^{+,\circ}})\tens\cor_{\opb{q_2}I^\circ_a}\br\simeq F\conv \cor_{\Delta^\circ_a}[-1].
\eneq
By the same proof as for \eqref{eq:iso1} we have reduced to showing:
\eq\label{eq:iso4}
&& \reim{p_{12}} \cor_{\Delta^{+,\circ}\cap\opb{p_{23}}\opb{q_2}I^\circ_a}\simeq\cor_{\Delta^\circ_a}[-1].
\eneq
This isomorphism is deduced from~\eqref{eq:iso2b} by duality.
\end{proof}

\begin{lemma}\label{le:3}
Let $F,G\in\Derb_\Rc(\cor_{\fX})$ and let $0\leq a<\alpha_X$. Then there exists $c>a$ such that for $a\leq b\leq c$, \eqref{eq:mor1} and \eqref{eq:mor2} are isomorphisms. 
\end{lemma}

\begin{proof}
(i) Let us treat \eqref{eq:mor1}.
 Set $H\eqdot \roim{q_2}\rhom(F\conv \cor_{\Delta^+},\epb{q_1}G)$.  Since $F, G \in \Derb_\Rc(\cor_{\fX})$, then $H\in\Derb_\Rc(\cor_\R)$ by \cite{Sc20}*{\S 2.} and therefore by \cite{KS90}*{Lem.~8.4.7}, 
\eqn
\rsect_{I_a}(\R;H)\isoto \rsect_{I_b}(\R;H)
\eneqn
for $a\leq b\leq c$ for some $c>a$. Applying Lemma~\ref{le:2} (i), we get the result.

\spa
(ii) Let us treat \eqref{eq:mor2}. First, using \cite{PS20}*{Prop.~2.2.3}, we obtain                
\eqn
\RHom(F, G\conv \cor_{\Delta_a}) &\simeq& \RHom(F \conv (\cor_{\Delta_a^\circ} \tens \opb{r_2}\omega_X),G) \\
& \simeq& \RHom((F \tens \, \omega_X) \conv \cor_{\Delta_a^\circ}, G).
\eneqn

Second, set $H \eqdot \roim{q_2}\rhom((F \tens \, \omega_X)\conv \cor_{\Delta^{+,\circ}},\epb{q_1}G))$. Since $F, G \in \Derb_\Rc(\cor_{\fX})$, then $H\in\Derb_\Rc(\cor_\R)$ and therefore by \cite{KS90}*{Lem.~8.4.7}
\eqn
\rsect(I^\circ_a;H)\isoto \rsect(I^\circ_b;H)
\eneqn
for $a\leq b\leq c$ for some $c>a$. Applying Lemma~\ref{le:2} (ii), we get the result.

\end{proof}

\begin{theorem}\label{th:main}
Let $\fX$ be a b-analytic manifold endowed with a good distance and satisfying~\eqref{hyp:dist2}. Let $F,G\in\Derb_\Rc(\cor_{\fX})$. If  $\dist(F,G)\leq a$  with $0\leq a<\alpha_X$, then $F$ and $G$ are $a$-isomorphic. In particular, if $\dist(F,G)=0$,  then $F$ and $G$ are isomorphic. 
\end{theorem}
\begin{proof}
The proof proceeds in four steps.  Let $0\leq a\leq b$ with $b$ small enough so that Lemma \ref{le:3} holds.

\spa
(i) Consider Diagram~\eqref{diag:1} below

\eq\label{diag:1}
&&\ba{l}\xymatrix{ \ar @{} [dr] |{\raisebox{.5pt}{\textcircled{\raisebox{-.9pt} {1}}}}
\Hom(F\conv \cor_{\Delta_{2a}},G\conv \cor_{\Delta_a})\times\Hom(G\conv \cor_{\Delta_a},F)\ar[r]^-\conv\ar[d]&
\Hom(F\conv \cor_{\Delta_{2a}},F)\ar[d]\\
\ar @{} [dr] |{\raisebox{.5pt}{\textcircled{\raisebox{-.9pt} {2}}}}
\Hom(F\conv \cor_{\Delta_{a+b}},G\conv \cor_{\Delta_b})\times\Hom(G\conv \cor_{\Delta_{b}},F\conv \cor_{\Delta_{b-a}})\ar[d] \ar[r]^-\conv&\Hom(F\conv \cor_{\Delta_{a+b}},F\conv \cor_{\Delta_{b-a}}) \ar[d]\\
\ar @{} [dr] |{\raisebox{.5pt}{\textcircled{\raisebox{-.9pt} {3}}}}
\Hom(F\conv \cor_{\Delta_{a+b}},G\conv \cor_{\Delta_b})\times\Hom(G\conv \cor_{\Delta_b},F)\ar[r]^-\conv\ar[d]&
\Hom(F\conv \cor_{\Delta_{a+b}},F)\ar[d]\\
\Hom(F\conv \cor_{\Delta_{2b}},G\conv \cor_{\Delta_b})\times\Hom(G\conv \cor_{\Delta_b},F)\ar[r]^-\conv&
\Hom(F\conv \cor_{\Delta_{2b}},F)
}\ea\eneq
Let us show that this diagram commutes.

\spa

 Diagram \circlednb{1} is obtained by applying the functor $\cdot \conv \cor_{\Delta_{b-a}}$ to the first line.

Diagram \circlednb{2} is obtained by composing the second line with the canonical morphism $F \conv \cor_{\Delta_{b-a}} \to F$.

Diagram \circlednb{3} is obtained by composing the third line with the canonical morphism $F \conv \cor_{\Delta_{2b}} \to F \conv \cor_{\Delta_{a+b}}$.

Hence, Diagram ~\eqref{diag:1} commutes.

\spa 
(ii) Assume $\dist(F,G)\leq a$, let $b>a$  and let $u_b\in\Hom(F\conv\cor_{\Delta_b},G)$ and $v_b\in\Hom(G\conv\cor_{\Delta_b},F)$ be as in Definition~\ref{def:1}. Denote by $\epsilon_{b}$ the natural morphism associated with $\cor_{\Delta_b}\to\cor_{\Delta}$, by $\Psi_b \cl \Derb(\cor_X) \to \Derb(\cor_X)$ the functor $L \mapsto L \conv \cor_{\Delta_b}$ and by $\Psi_F \cl \Derb(\cor_{X \times X}) \to \Derb(\cor_{X})$, $K \mapsto F \conv K$.  Then, the $b$-isomorphism equations are explicitely given by
\eq\label{eq:uconvv}
&&v_b\conv  \Psi_b(u_b)=  \Psi_F(\epsilon_{2b}).
\eneq

\spa
(iii) The vertical arrows in Diagram \circlednb{1} are isomorphisms thanks to~\eqref{hyp:dist2}  (see~\cite{PS20}*{Prop.~2.2.3}). The vertical arrows in Diagram \circlednb{2} and \circlednb{3} are isomorphisms thanks to Lemma~\ref{le:3}.

\spa
(iv)  Hence, there exist $u_a\in\Hom(F\conv\cor_{\Delta_a},G)$ and $v_a\in\Hom(G\conv\cor_{\Delta_a},F)$ whose images are the morphisms $u_b$ and $v_b$. Using the vertical isomorphism, equation \eqref{eq:uconvv} translates to the same equation with $b$ replaced by $a$. The same result holds with $F$ and $G$ interchanged. This complete the proof.
\end{proof}

\begin{corollary}
Let $X$ be a real analytic manifold  endowed with a good distance and satisfying~\eqref{hyp:dist2}. Let $F,G\in\Derb_\Rc(\cor_X)$, both with compact support. 
 Assume that $\dist(F,G)=0$. Then $F\simeq G$. 
\end{corollary}
\begin{proof}
Let $Y$ be an open subanalytic subset of $X$ containing the supports of $F$ and $G$. Then regard $\fY=(Y,X)$ as a b-analytic manifold and apply Theorem~\ref{th:main}. 
\end{proof}

\begin{remark}
When the space $X$ is a finite dimensional real vector space endowed with a closed proper convex subanalytic cone $\gamma$ with nonempty interior and a vector $v$ in the interior of $\gamma$, then, thanks to \cite{BP21}*{Cor.~5.9}, Theorem \ref{th:main} implies the same results for $\gamma$-sheaves endowed with the interleaving distance associated with the pair $(\gamma,v)$ (see \cite{BP21}*{Def.~4.8}) .
\end{remark}

\providecommand{\bysame}{\stLeavevmode\hbox to3em{\hrulefill}\thinspace}

\vspace*{1cm}
\noindent
\begin{tabular}{cc}
\parbox[t]{14em}
{\scriptsize{
Francois Petit \\
Université de Paris\\ 
CRESS, INSERM, INRA\\ 
F-75004 Paris France\\
e-mail address: francois.petit@u-paris.fr}}
& 
\parbox[t]{14em}
{\scriptsize{
		Pierre Schapira\\
		Sorbonne Universit{\'e}, CNRS IMJ-PRG\\
		4 place Jussieu, 75252 Paris Cedex 05 France\\
		e-mail: pierre.schapira@imj-prg.fr\\
		http://webusers.imj-prg.fr/\textasciitilde pierre.schapira/
}}
\\
&\\
\parbox[t]{14em}
{\scriptsize{
		Lukas Waas\\
		Ruprecht-Karls-Universität Heidelberg\\
		Mathematisches Institut \\
		Im Neuenheimer Feld 205, 69120 Heidelberg, Deutschland\\
		e-mail:lwaas@mathi.uni-heidelberg.de\\
}}&
\end{tabular}
\end{document}